\documentclass[12pt,amscd]{amsart}
\usepackage{tikz}
\usetikzlibrary{decorations.markings}
\tikzset{
  mid arrow/.style={
    postaction={
      decorate,
      decoration={
        markings,
        mark=at position 0.6 with {\arrow{>}}
      }
    }
  }
}
\usepackage[all]{xy}
\usepackage{caption} 
\usepackage{amsmath,amsxtra,amssymb,latexsym, amscd,amsthm}
\usepackage{indentfirst}
\usepackage[mathscr]{eucal}
\usepackage[pagebackref=true]{hyperref}

\newtheorem{thm}{Theorem}[section]
\newtheorem{cor}[thm]{Corollary}
\newtheorem{lem}[thm]{Lemma}

\theoremstyle{definition}
\newtheorem{defn}[thm]{Definition}

\numberwithin{equation}{section}

\DeclareMathOperator{\NN}{\mathbb {N}}
\DeclareMathOperator{\ZZ}{\mathbb {Z}}
\DeclareMathOperator{\RR}{\mathbb {R}}

\DeclareMathOperator{\supp}{supp}
\DeclareMathOperator{\NP}{NP}

\def\a {\mathbf a}
\def\b {\mathbf b}
\def\c {\mathbf c}

\def\e {\mathbf e}

\def\u {\mathbf u}
\def\v {\mathbf v}
\def\w {\mathbf w}
\def\x {\mathbf x}

\def\k {\mathrm k}

\begin{document}

\title[Normal edge-weighted graphs]{Normal edge ideals of edge-weighted graphs}

\author{Thanh Vu}
\address{Institute of Mathematics, VAST, 18 Hoang Quoc Viet, Hanoi, Vietnam}
\email{vuqthanh@gmail.com}

\author{Guangjun Zhu}
\address{School of Mathematical Sciences, Soochow University, Suzhou, Jiangsu, 215006, P.R. China}
\email{zhuguangjun@suda.edu.cn}


\date{}

\begin{abstract}
    We classify all normal edge ideals of edge-weighted graphs.
\end{abstract}

\maketitle

\section{Introduction}
\label{sect_intro}
Let $R = \k[x_1,\ldots,x_n]$ be a standard graded polynomial ring over a field $\k$. In \cite{MVZ}, Minh, Vu, and Zhu classified all integrally closed and normal edge ideals of weighted oriented graphs. In this work, we use the results developed in \cite{MVZ} to classify all normal edge ideals of edge-weighted graphs. Recall that a non-zero homogeneous ideal $I$ of $R$ is normal if $I^t = \overline{I^t}$ for all natural exponents $t \ge 1$, where $\overline{I^t}$ denotes the integral closure of $I^t$. We refer to \cite{SH} for criteria for the normality of general homogeneous ideals. Even when $I$ is a monomial ideal, determining the normality of $I$ is a difficult problem and has a deep connection to the theory of linear and integer programming \cite{DV, HL, HT, Sch1, Sch2}. 

Let us now recall the notion of edge ideals of edge-weighted graphs introduced by Paulsen and Sather-Wagstaff \cite{PS}. Let $G$ be a simple graph on the vertex set $V(G)=[n] = \{1,\ldots,n\}$ and the edge set $E(G)$. Assume that $\w:E(G)\rightarrow \ZZ_{>0}$ is a weight function on the edges of $G$. The edge ideal of the edge-weighted graph $(G,\w)$ is defined by
$$I(G,\w) = \big( (x_i x_j)^{\w(e)} \mid  e = \{i,j\} \in E(G)\big) \subseteq S=\k[x_1,\ldots,x_n].$$
If $\w(e) = 1$ for all edge $e \in E(G)$, then $I(G,\w) = I(G)$ is the usual edge ideal of $G$. For simplicity of notation, we call an edge-weighted graph $(G,\w)$ integrally closed (normal) if its edge ideal is. Ohsugi and Hibi \cite{OH} and Simis, Vasconcelos, and Villarreal \cite{SVV} computed the normalization of the edge ring of a simple graph. One can then deduce that the edge ideal of a simple graph $G$ is normal if and only if $G$ satisfies the odd cycle conditions. In terms of forbidden subgraphs, it says that $I(G)$ is normal if and only if $G$ does not have the disjoint union of two odd cycles as an induced subgraph. A squarefree monomial ideal is integrally closed. In particular, edge ideals of simple graphs are integrally closed. Nonetheless, most edge-weighted graphs are not integrally closed. Duan, Zhu, Cui, and Li \cite{DZCL} proved that $I(G,\w)$ is integrally closed if and only if $(G,\w)$ does not have the following configurations as induced weighted subgraphs, where the weights are all nontrivial, i.e., have weights greater than $1$.

\begin{center}
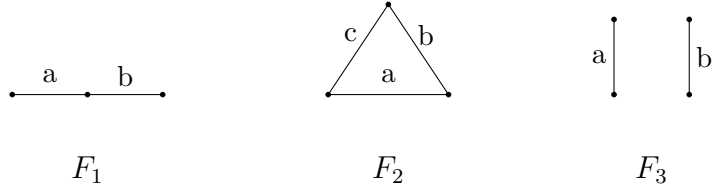

\begin{tikzpicture}

\draw[black, fill=black] (0,0) circle [radius=0.3mm];
\draw[black, fill=black] (1,0) circle [radius=0.3mm];
\draw[black, fill=black] (2,0) circle [radius=0.3mm];
\draw[-] (0,0) -- (1,0);
\draw[-] (1,0) -- (2,0);
\node at (0.5,0.25){\small a};
\node at (1.5,0.25){\small b};
\node at (1,-1) {$F_1$};

\draw[black, fill=black] (4.2,0) circle [radius=0.3mm];
\draw[black, fill=black] (5.8,0) circle [radius=0.3mm];
\draw[black, fill=black] (5,1.2) circle [radius=0.3mm];
\draw[-] (4.2,0) -- (5,1.2);
\draw[-] (5.8,0) -- (4.2,0);
\draw[-] (5.8,0) -- (5,1.2);
\node at (5,0.25){\small a};
\node at (5.5,0.8){\small b};
\node at (4.5,0.8){\small c};

\node at (5,-1) {$F_2$};

\draw[black, fill=black] (8,0) circle [radius=0.3mm];
\draw[black, fill=black] (9,0) circle [radius=0.3mm];
\draw[black, fill=black] (8,1) circle [radius=0.3mm];
\draw[black, fill=black] (9,1) circle [radius=0.3mm];

\draw[-] (8,0) -- (8,1);
\draw[-] (9,0) -- (9,1);
\node at (7.8,0.5){\small a};
\node at (9.2,0.5){\small b};
\node at (8.5,-1) {$F_3$};

\end{tikzpicture}
\captionof{figure}{Forbidden configurations for integrally closed edge-weighted graphs} 
  \label{fig:for1}
\end{center}

The normality of $I(G,\w)$ imposes further restrictions on $(G,\w)$. In the following figure, a triangle with a dashed line represents an induced odd cycle. When we do not specify the weight of the edge, we mean the weight could be any arbitrary positive integer. In particular, $F_4$ represents the disjoint union of an odd cycle with a nontrivial edge. $F_5$ represents two odd cycles with possibly only nontrivial edges connecting them.

\medskip

\begin{center}
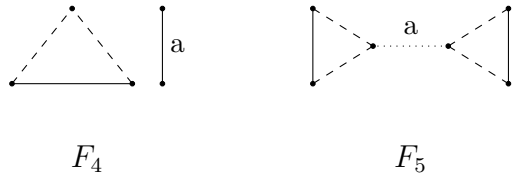

\begin{tikzpicture}

\draw[-] (0,0) -- (1.6,0);
\draw[dashed] (0,0) -- (0.8,1);
\draw[dashed] (1.6,0) -- (0.8,1);
\draw[-] (2,0) -- (2,1);
\draw[black, fill=black] (0,0) circle [radius=0.3mm];
\draw[black, fill=black] (1.6,0) circle [radius=0.3mm];
\draw[black, fill=black] (2,0) circle [radius=0.3mm];
\draw[black, fill=black] (0.8,1) circle [radius=0.3mm];
\draw[black, fill=black] (2,1) circle [radius=0.3mm];

\node at (2.2,0.5) {\small a};
\node at (1,-1) {$F_4$};

\draw[-] (4,0) -- (4,1);
\draw[dashed] (4,0) -- (4.8,0.5);
\draw[dashed] (4,1) -- (4.8,0.5);
\draw[dashed] (5.8,0.5) -- (6.6,1);
\draw[dashed] (5.8,0.5) -- (6.6,0);
\draw[dotted] (5.8,0.5) -- (4.8,0.5);
\draw[-] (6.6,1) -- (6.6,0);
\draw[black, fill=black] (4,0) circle [radius=0.3mm];
\draw[black, fill=black] (4,1) circle [radius=0.3mm];
\draw[black, fill=black] (4.8,0.5) circle [radius=0.3mm];
\draw[black, fill=black] (5.8,0.5) circle [radius=0.3mm];
\draw[black, fill=black] (6.6,1) circle [radius=0.3mm];
\draw[black, fill=black] (6.6,0) circle [radius=0.3mm];

\node at (5.3,0.75) {\small a};
\node at (5.3,-1) {$F_5$};

\end{tikzpicture}
\captionof{figure}{Further forbidden configurations for normal edge-weighted graphs} 
  \label{fig:for2}
\end{center}

\begin{thm}\label{thm_normality} Let $(G,\w)$ be an arbitrary edge-weighted graph. Then $I(G,\w)$ is normal if and only if  $(G,\w)$ does not have $F_1,\ldots,F_{5}$ as an induced edge-weighted subgraph.
\end{thm}
Let us outline the idea of the proof of Theorem \ref{thm_normality}. First, for each configuration $F_1,\ldots,F_5$, we give an explicit monomial in the integral closure of a certain power of $I(G,\w)$ but not in that power itself. Now, assume that $(G,\w)$ does not have $F_1,\ldots,F_5$ as an induced edge-weighted subgraph. We follow the steps below to prove the normality of $I(G,\w)$ by induction on the number of vertices and the number of edges of $G$.
\begin{enumerate}
    \item Reduce to the case where $G$ is connected and $(G,\w)$ has no leaf edges with trivial weights.
    \item Prove that dropping a nontrivial edge of $(G,\w)$ results in a graph that does not have $F_1,\ldots,F_5$ as an induced edge-weighted subgraph.
    \item Reduce to the case where $G$ has only one nontrivial edge.
    \item Reduce to the case where $G$ has no even cycles. In other words, $G$ is compact in the sense of Wang and Lu \cite{WL}.
    \item Establish the normality for each connected compact graph having only one nontrivial edge and no $F_4,F_5$ as induced edge-weighted subgraphs.
\end{enumerate}
We would like to note that the steps are similar to those in \cite{MVZ}, though the combinatorics of the edge-weighted graphs and weighted oriented graphs are subtly different. Furthermore, the proofs in both cases are quite technical with many details so we decide to separate this work from \cite{MVZ} to better present the results.

\section{Integrally closed edge ideals of edge-weighted graphs}\label{sec_intclosed}

Throughout the paper, we denote by $S = \k[x_1,\ldots,x_n]$ a standard graded polynomial ring over an arbitrary field $\k$. We first recall the definition and some properties of the integral closure of monomial ideals; see \cite{MV} for more details. We then give a simple proof of a result of \cite{DZCL}, which classifies integrally closed edge-weighted graphs. 

\begin{defn} Let $I$ be a monomial ideal of $S$. The exponent set of $I$ is $E(I) = \{ \a \in \NN^n \mid x^\a \in I\}$. The Newton polyhedron of $I$, denoted by $\NP(I)$, is the convex hull of the exponent set of $I$ in $\RR^n$.
\end{defn}
The following result \cite[Exercise 4.23]{Ei} is standard.
\begin{lem}\label{lem_newton_polyhedra} Let $I$ be a monomial ideal of $S$. The integral closure of $I$ is a monomial ideal with exponent set $E(\overline{I}) = \NP(I) \cap \ZZ^n$.    
\end{lem} 
For $\a = (a_1, \ldots, a_n),\b = (b_1, \ldots, b_n)\in \RR^n$, the notation $\a \ge \b$ means that $a_j \ge b_j$ for all $j = 1, \ldots, n$.
\begin{lem}\label{lem_criterion_integral} Let $I$ be a monomial ideal and $\a \in \NN^n$ be an exponent. Then $x^\a \in \overline{I}$ if and only if there exist nonnegative real numbers $c_j$ and exponents $\b_j$ of $E(I)$ corresponding to minimal generators of $I$ for $j = 1, \ldots, s$ such that $\sum_{j=1}^s c_j \ge 1$ and $\a \ge \sum_{j = 1}^s c_j \b_j$. In particular, when $x^\a$ is a minimal generator of $\overline{I}$, we can choose $c_j$ such that $\sum_{j=1}^s c_j =1$.
\end{lem}
\begin{proof}
    The conclusion follows from Lemma \ref{lem_newton_polyhedra} and the fact that for exponents $\a,\b \in \NN^n$ and a monomial ideal $I$, $\a \ge \b$ and $x^\b \in I$ implies that $x^\a \in I$.
\end{proof}
For a monomial $f$ in $S$, the support of $f$, denoted by $\supp (f)$, is the set of all indices $i \in [n]$ such that $x_i \mid f$. For a monomial ideal $J$ of $S$ and a subset $V$ of $[n]$, the restriction of $J$ to $V$, denoted by $J_V$, is 
$$J_V = (f \mid f \text{ is a minimal generator of } J \text{ such that } \supp f \subseteq V).$$

We have 
\begin{cor}\label{cor_restriction}
    Let $I$ be a monomial ideal. Then 
    $$ \overline{ I_V } = (\overline{I})_V.$$ 
\end{cor}

We also fix the following notation throughout the paper. Let $G$ be a finite simple graph on the vertex set $V(G) = [n]$ and the edge set $E(G)$. For a vertex $i \in V(G)$, the neighborhood of $i$ is defined by 
$$N_G(i) = \{j \in V(G) \mid \{i,j\} \in E(G)\}.$$
The degree of a vertex $i$ is the cardinality of its neighborhood. A vertex of degree $1$ is called a leaf of $G$. Let $\w: E(G) \to \ZZ_+$ be a weight function on the edges of $G$. An edge $e \in E(G)$ is called a nontrivial edge if $\w(e) > 1$.

We now prove that the configurations $F_1, F_2, F_3$ give rise to non-integrally closed edge-weighted graphs. We denote by $\e_1,\ldots,\e_n$ the canonical basis of $\RR^n$.

\begin{lem}\label{lem_F_1} Assume that $I(G,\w) = (x_1^ax_2^a,x_2^bx_3^b)$ with $a,b> 1$. Then, $I(G,\w)$ is not integrally closed. 
\end{lem}
\begin{proof} We may assume that $a\le b$. We have $\u = (a,a,0)$, $\v = (0,b,b)  \in \NP(I(G,\w))$. Let 
$$\w = \frac{1}{2} (\u + \v) = \left ( \frac{a}{2}, \frac{a+b}{2}, \frac{b}{2} \right ).$$
By Lemma \ref{lem_criterion_integral} and the assumption that $a,b > 1$, we deduce that $x_1^{a-1} x_2^b x_3^{b-1} \in \overline{I(G,\w)}$. The conclusion follows.
\end{proof}

\begin{lem}\label{lem_F_2} Assume that $I(G,\w) = (x_1^ax_2^a,x_2^bx_3^b,x_1^cx_3^c)$ with $a,b,c > 1$. Then, $I(G,\w)$ is not integrally closed.
\end{lem}
\begin{proof} We may assume that $a\le b\le c$. We have $\u = (a,a,0)$, $\v = (0,b,b)$ and $\w = (c,0,c) \in \NP(I(G,\w))$. Let 
$$\x = \frac{1}{2} (\u + \v) = \left (\frac{a}{2}, \frac{a+b}{2}, \frac{b}{2} \right ).$$
Since $a,b \ge 2$, by Lemma \ref{lem_criterion_integral}, we deduce that $f = x_1^{a-1} x_2^{b} x_3^{b-1} \in \overline{I(G,\w)}$. Since $c \ge a,b$, $f \notin I(G,\w)$. The conclusion follows.   
\end{proof}

\begin{lem}\label{lem_F_3} Assume that $I(G,\w) = (x_1^ax_2^a,x_3^bx_4^b)$ with $a,b> 1$. Then, $I(G,\w)$ is not integrally closed. 
\end{lem}
\begin{proof} We have $\u = (a,a,0,0)$ and $\v = (0,0,b,b)  \in \NP(I(G,\w))$. Let  
$$\w = \frac{1}{2} (\u + \v) = \left ( \frac{a}{2}, \frac{a}{2},\frac{b}{2}, \frac{b}{2} \right ).$$
Since $a,b > 1$, by Lemma \ref{lem_criterion_integral}, we deduce that $x_1^{a-1} x_2^{a-1} x_3^{b-1} x_4^{b-1} \in \overline{I(G,\w)}$. The conclusion follows.
\end{proof}

\begin{thm}\label{thm_intclosed} Let $(G,\w)$ be an edge-weighted graph. Then $I(G,\w)$ is integrally closed if and only if $G$ does not have $F_1,F_2,F_3$ as an induced edge-weighted subgraph. 
\end{thm}
\begin{proof} By Lemmas \ref{lem_F_1}, \ref{lem_F_2}, and \ref{lem_F_3}, we may assume that $(G,\w)$ does not have $F_1,F_2,F_3$ as induced edge-weighted subgraphs. We need to prove that $\overline{I(G,\w)} = I(G,\w)$. We prove this by induction on $|V(G)|$. The base case where $|V(G)| \le 2$ is clear. Hence, we may assume that $|V(G)| \ge 3$. Let $f=x^\a$ be a minimal generator of $\overline{I(G,\w)}$. By induction and Corollary \ref{cor_restriction}, we may assume that $G$ does not have isolated vertices and $\supp f = [n]$. 

Since $n \ge 3$ and $(G,\w)$ has no $F_1,F_2,F_3$ as an induced weighted subgraph, we deduce that $(G,\w)$ must have an edge with trivial weight. Let $\{i,j\}$ be an edge of $G$ with trivial weight. Then $x_ix_j \in I(G,\w)$. Since $\supp f = [n]$, $f$ is divisible by $x_ix_j$. Hence, $f \in I(G,\w)$. The conclusion follows.
\end{proof}

\section{Normal edge ideals of edge-weighted graphs}\label{sec_normal}
In this section, we prove Theorem \ref{thm_normality}. First, we have some simple lemmas.

\begin{lem}\label{lem_criterion_integral_k} Let $I\subseteq S$ be a monomial ideal and $\a \in \NN^n$ be an exponent. Then $x^\a \in \overline{I^k}$ if and only if there exist nonnegative real numbers $c_j$ and exponents $\b_j$ of $E(I)$ for $j = 1, \ldots,s$ such that $\sum_{j=1}^s c_j \ge k$ and $\a \ge \sum_{j=1}^s c_j \b_j$. In particular, when $x^\a$ is a minimal generator of $\overline{I^k}$, we can choose $c_j$ such that $\sum_{j=1}^s c_j =k$.
\end{lem}
\begin{proof}
    The conclusion follows from Lemma \ref{lem_newton_polyhedra}, Lemma \ref{lem_criterion_integral}, and the fact that a vector $\v \in \RR^n$ belongs to $\NP(I^k)$ if and only if $\frac{1}{k}\v \in \NP(I)$.
\end{proof}

\begin{lem}\label{lem_res_normal} Let $I\subseteq S$ be a monomial ideal and $V \subseteq [n]$. Assume that $I$ is normal. Then $I_V$ is normal.
\end{lem}
\begin{proof}
    The conclusion follows from Corollary \ref{cor_restriction} and the fact that $(I^k)_V = (I_V)^k$.
\end{proof}

We now prove that the configurations $F_4$ and $ F_{5}$ give rise to non-normal edge-weighted graphs.

\begin{lem}\label{lem_F_4} Assume that $(G,\w)$ is isomorphic to the configuration $F_4$. Then, $I(G,\w)$ is not normal.    
\end{lem}
\begin{proof}
By Lemma \ref{lem_F_3}, we may assume that all edges of the odd cycle $C$ in $G$ have trivial weights. In other words, we may assume that 
$$I = I(G,\w) = (x_1x_2,\ldots,x_{2k-2}x_{2k-1},x_1x_{2k-1}) + (x_{2k+1}^ax_{2k}^a),$$
where $k,a$ are integers at least $2$. We will prove that $I^{k}$ is not integrally closed. Let 
$$\u_1 = \e_1 + \e_2,\ldots, \u_{2k-2} = \e_{2k-2} + \e_{2k-1}, \u_{2k-1} = \e_{2k-1} + \e_1, \u_{2k} = a\e_{2k} + a \e_{2k+1}$$
be the vectors in the Newton polyhedron of $I$ corresponding to minimal generators of $I$. We have 
$$\v = \frac{1}{2}(\u_1 + \cdots + \u_{2k}) = (1,\ldots,1,\frac{a}{2},\frac{a}{2}).$$
Since $a \ge 2$, by Lemma \ref{lem_criterion_integral_k}, we deduce that $f = x_1 \cdots x_{2k-1} x_{2k}^{a-1}x_{2k+1}^{a-1} \in \overline{I^k}.$ Since there is no edge between $j$ and $2k$ or $2k+1$ for any $j \in \{1,\ldots,2k-1\}$, we deduce that $f\notin I^k$. The conclusion follows.
\end{proof}

\begin{lem}\label{lem_F_5}  Assume that $(G,\w)$ is isomorphic to the configuration $F_{5}$. Then $I(G,\w)$ is not normal.    
\end{lem}
\begin{proof} By Lemmas \ref{lem_F_1} and \ref{lem_F_3}, we may assume that all edges on the two odd cycles have trivial weights. In other words, we may assume that 
\begin{align*}
    I = I(D,\w) &= (x_1x_2,\ldots,x_{2k-2}x_{2k-1},x_1x_{2k-1}) \\
    &+ (x_{2k}x_{2k+1}, \ldots,x_{2k+2l-1}x_{2k+2l},x_{2k}x_{2k+2l}) + ((x_ix_j)^{w_{ij}} \mid \{i,j\} \in W),
\end{align*}  
where $k \ge 2$, $l$ are positive integers and $W \subseteq E(G)$ is a possibly empty collection of edges $\{i,j\}$ with $i \in \{1,\ldots,2k-1\}$ and $j\in \{2k,\ldots,2k+2l\}$, and $w_{ij} > 1$ for all $\{i,j\}\in W$. Let 
\begin{align*}
    \u_1 &= \e_1 + \e_2,\ldots, \u_{2k-1} = \e_{2k-1} + \e_{1}, \\
    \u_{2k} &= \e_{2k} + \e_{2k+1}, \ldots, \u_{2k+2l} = \e_{2k} + \e_{2k+2l}
\end{align*}
be the vectors in the Newton polyhedron of $I$ corresponding to minimal generators of $I$. We have 
$$\v = \frac{1}{2}(\u_1 + \cdots + \u_{2k+2l}) = (1,\ldots,1,1,\ldots,1).$$
By Lemma \ref{lem_criterion_integral_k}, we deduce that $f = x_1 \cdots x_{2k+2l} \in \overline{I^{k+l}}.$ Since every possible edge connecting the two odd cycles has nontrivial weight, we deduce that $f \notin I^{k+l}$. The conclusion follows.
\end{proof}

The following three lemmas are analogous to the corresponding lemmas in \cite{MVZ} and allow us to achieve the first two reduction steps outlined in the Introduction.

\begin{lem}\label{lem_drop_bipartite_component} Assume that $G = G_1 \cup G_2$ is the disjoint union of two subgraphs $G_1$ and $G_2$ where $G_2$ is bipartite. Let $\w$ be a weight function on $E(G)$ such that $\w(e) = 1$ for all $e \in E(G_2)$. Denote by $\w_1$ the induced weight function of $\w$ on $E(G_1)$. Then $I(G,\w)$ is normal if and only if $I(G_1,\w_1)$ is normal.    
\end{lem}
\begin{proof} The conclusion follows from Lemma \ref{lem_res_normal} and \cite[Theorem 2.1]{MT}. See also \cite[Corollary 2.8]{BH}.
\end{proof}

The following result is a consequence of \cite[Theorem 2.1]{ANKRQ}. We give an alternative proof using our approach.
\begin{lem}\label{lem_drop_leaf} Assume that $v$ is a leaf of $G$ and $\w(\{u,v\}) = 1$, where $u$ is the unique neighbor of $v$. Let $G' = G\backslash v$ be the induced subgraph of $G$ on $V(G) \backslash v$ and $\w'$ be the induced weight function of $\w$ on $E(G')$. Then $I(G,\w)$ is normal if and only if $I(G',\w')$ is normal.     
\end{lem}
\begin{proof}
    By Lemma \ref{lem_res_normal}, it suffices to prove that if $I(G',\w')$ is normal then $I(G,\w)$ is normal. For simplicity of notation, we denote by $J = I(G',\w')$ and $K = I(G,\w)$. We prove by induction on $t$ that $K^t = \overline{K^t}$. The base case $t = 1$ follows from Theorem \ref{thm_intclosed}. Let $f =x^\a \in \overline{K^t}$ be a minimal generator. By induction on the number of vertices and Corollary \ref{cor_restriction}, we may assume that $\supp (f) = [n]$. By Lemma \ref{lem_criterion_integral_k}, there exist nonnegative real numbers $\beta_j$ for $j = 1, \ldots, s$ such that $\sum_1^s \beta_j = t$ and  
    $$\a \ge \b = \sum_{j=1}^s \beta_j \b_j,$$
    where $\b_j=(b_{j,1},\ldots,b_{j,n})$ are exponents corresponding to some minimal generators of $K$. We may assume that $v =n$, $u = n-1$ and $\b_s = \e_{n-1} + \e_n$ be the exponent corresponding to the generator $x_{n-1}x_n$ of $K$. By induction on $t$, we may assume that $\beta_j < 1$ for all $j = 1,\ldots,s$. If $\beta_s = 0$ then $f\in \overline{J^t}$. By assumption, we deduce that $f\in J^t \subseteq K^t$. Hence, we may assume that $\beta_s > 0$. By the proof of \cite[Lemma 2.9]{MV}, we deduce that $a_n = 1$. We can further assume that $\b_{u},\ldots,\b_{s-1}$ are all the other exponents such that $b_{j,n-1} > 0$. There are two cases.

\medskip

\noindent \textbf{Case 1.} $\sum_{j=u}^{s-1} \beta_j < 1 - \beta_s$. In particular, $\sum_{j=1}^{u-1} \beta_j > t-1$. Let $\a' = \a - \b_s$. Since $b_{j,n-1} = b_{j,n} = 0$ for all $j = 1, \ldots,u-1$, we have $\a' \ge \sum_{j=1}^{u-1} \beta_j \b_j$. By Lemma \ref{lem_criterion_integral_k}, we deduce that $x^{\a'} \in \overline{J^{t-1}} = J^{t-1}$. Hence, $x^\a = x^{\a'} (x_{n-1}x_n) \in K^t$.

\medskip

\noindent \textbf{Case 2.} $\sum_{j=u}^{s-1} \beta_j > 1 - \beta_s$. Then there exist nonnegative real numbers $\gamma_j$ for $j = u, \ldots,s-1$ such that $\sum_{j=u}^{s-1} \gamma_j = 1 - \beta_s$. Let $$\b' = \sum_{j=1}^{u-1} \beta_j \b_j + \sum_{j=u}^{s-1} (\beta_j - \gamma_j) \beta_j,\; \b^{''} = \sum_{j=u}^{s-1} \gamma_j \beta_j + \beta_s \b_s.$$ 
Then we have $\b = \b' + \b^{''}$ where $\b' = (b'_1,\ldots,b'_n)$ with $b'_{n-1} \le b_{n-1} - 1$ and $b'_{n} = 0$. Let $\a' = \a - \b_s$. Then, we have $\a' \ge \b'$. By Lemma \ref{lem_criterion_integral_k}, we deduce that $x^{\a'} \in \overline{J^{t-1}} = J^{t-1}$. Hence, $x^{\a} = x^{\a'} (x_{n-1}x_n) \in K^t$. The conclusion follows.
\end{proof}

In the sequel, when $(G,\w)$ has an induced subgraph isomorphic to $F_i$ for some $i = 1, \ldots, 5$ we say that $(G,\w)$ has an $F_i$-minor.

\begin{lem}\label{lem_drop_non_trivial_edge} Let $(G,\w)$ be an edge-weighted graph. Assume that $(G,\w)$ has no $F_i$-minor for all $i = 1, \ldots,5$ and $e\in E(G)$ is an edge of $G$ with $\w(e) > 1$. Let $G' = G\setminus e$, the subgraph of $G$ obtained by removing the edge $e$, and $\w'$ be the induced weight function on $E(G')$. Then $(G',\w')$ has no $F_i$-minor for all $i = 1, \ldots,5$.
\end{lem}
\begin{proof} Assume by contradiction that $(G',\w')$ has an $F_i$-minor for some $i = 1, \ldots, 5$. It is easy to see the following.

\medskip

If $(G',\w')$ has an $F_1$-minor, then $(G,\w)$ has an $F_1$-minor or an $F_2$-minor.

\medskip

If $(G',\w')$ has an $F_2$-minor, then $(G,\w)$ has an $F_2$-minor.

\medskip

If $(G',\w')$ has an $F_3$-minor, then $(G,\w)$ has an $F_1$-minor or an $F_3$-minor.

\medskip

If $(G',\w')$ has an $F_4$-minor, then $(G,\w)$ has an $F_1$-minor or an $F_4$-minor.

\medskip

If $(G',\w')$ has an $F_5$-minor, then $(G,\w)$ has an $F_5$-minor. 

In all cases, $(G,\w)$ has an $F_i$-minor for some $i = 1, \ldots,5$, a contradiction. The conclusion follows.
\end{proof}

\begin{lem}\label{lem_drop_trivial_even_cycle} Let $G$ be an even cycle. We denote by $\b_1,\ldots, \b_{2k}$ the exponent vectors corresponding to the minimal generators of $I(G)$. Let $\beta_1,\ldots,\beta_{2k}$ be positive real numbers. Then there exist nonnegative real numbers $\gamma_1,\ldots,\gamma_{2k}$ such that $\sum_{j=1}^{2k} \gamma_j = \sum_{j=1}^{2k} \beta_j$, at least one of $\gamma_j$ is zero and $\sum_{j=1}^{2k} \gamma_j \b_j = \sum_{j=1}^{2k} \beta_j \b_j$.   
\end{lem}
\begin{proof} By the assumption, we deduce that 
$$\b_{1} + \b_3 + \cdots + \b_{2k-1} = \b_2 + \cdots + \b_{2k}.$$
We may assume that $\beta_{2k} = \min \{\beta_{2j} \mid j = 1, \ldots, k\}$. We have 
\begin{align*}
    \sum_{j=1}^{2k} \beta_j \b_j &=  \sum_{j=0}^{k-1} \beta_{2j+1} \b_{2j+1} + \sum_{j=1}^k \beta_{2j} \b_{2j} \\
    & = \sum_{j=0}^{k-1} (\beta_{2j+1} + \beta_{2k}) \b_{2j+1} + \sum_{j=1}^{k-1} (\beta_{2j} - \beta_{2k}) \b_{2j}.
\end{align*}
    The conclusion follows.
\end{proof}

\begin{lem}\label{lem_drop_non_trivial_even_cycle} Let $(G,\w)$ be an edge-weighted even cycle with exactly one nontrivial edge. We denote by $\b_1,\ldots, \b_{2k}$ the exponent vectors corresponding to the minimal generators of $I(G)$. Let $\beta_1,\ldots,\beta_{2k}$ be positive real numbers. Then there exist nonnegative real numbers $\gamma_1,\ldots,\gamma_{2k}$ such that $\sum_{j=1}^{2k} \gamma_j = \sum_{j=1}^{2k} \beta_j$, at least one of $\gamma_j$ is zero and $\sum_{j=1}^{2k} \gamma_j \b_j \le \sum_{j=1}^{2k} \beta_j \b_j$.   
\end{lem}
\begin{proof} We may assume that $\b_1 = a(\e_1 + \e_2)$ is the exponent corresponding to the only nontrivial edge of $G$. Let $\beta_{2j+1} = \min\{\beta_{2l+1} \mid l = 0,\ldots, k-1\}$ and

$$\gamma_{i}  = \begin{cases} \beta_i + \beta_{2j+1} & \text{ if } i \text{ is even},\\
\beta_i - \beta_{2j+1} & \text{ if } i \text{ is odd}.\end{cases}$$
Then $\sum_{i=1}^{2k} \gamma_i = \sum_{i=1}^{2k} \beta_i$, at least one of $\gamma_i$ is zero. Furthermore,

$$\sum_{i=1}^{2k} \gamma_i \b_i = \sum_{i=1}^{2k} \beta_i \b_i + \beta_{2j+1} \left ( \frac{1}{a}-1 \right )\b_{1} \le \b. $$
The conclusion follows.
\end{proof}

We now introduce the compact graphs defined by Wang and Lu \cite{WL} and their classification.
\begin{defn}
    A simple graph $G$ is compact if $G$ has no even cycles and satisfies the odd cycle conditions.
\end{defn}

The following classification of compact graphs is \cite[Theorem 2.12]{WL}. For convenience in our applications, we use different terminology from that of Wang and Lu. We call a finite collection of odd cycles sharing a common vertex a bouquet of odd cycles, and the common vertex is called its stem. In particular, an odd cycle itself is considered to be a special bouquet of odd cycles.

\begin{thm}\label{thm_compact_graphs} A connected simple graph $G$ without leaves is compact if and only if it belongs to one of the following classes of graphs.
\begin{enumerate}
    \item A bouquet of odd cycles.
    \item Two bouquets of odd cycles whose stems are joined by an edge and possibly another path of even length.
    \item Three bouquets of odd cycles whose stems form a triangle.
\end{enumerate}
\end{thm}

\begin{lem}\label{lem_normal_compact} Let $G$ be an edge-weighted connected compact graph and $\w$ be a weight function on the edges of $G$. Assume that $(G,\w)$ has only one nontrivial weight and has no $F_4,F_5$-minor. Then $I(G,\w)$ is normal.
\end{lem}
\begin{proof}We prove by induction on the number of vertices of $G$. The base case where $|V(G)|= 2$ is clear. By Lemma \ref{lem_drop_leaf}, we may assume that $(G,\w)$ has no leaf edges with trivial weights. Let $e =\{1,2\}$ be the unique edge of $G$ such that $\w(e)=a >1$. Let $f = x^\a$ be a minimal generator of $\overline{I^t}$ for some $t$. By Lemma \ref{lem_criterion_integral_k}, there exist nonnegative real numbers $\beta_1,\ldots,\beta_s$ and exponents $\b_1,\ldots,\b_s$ corresponding to minimal generators of $I(G,\w)$ such that 
$$\sum_{j=1}^s \beta_j = t\text{ and } \a \ge \b = \sum_{j=1}^s \beta_j \b_j.$$
We prove by induction on $t$ that $f \in I^t$. Since $(G,\w)$ has no $F_4$-minor, we deduce that $e$ contains a unique stem of $G$ and one of the following must hold:
\begin{enumerate}
    \item $e$ is adjacent to an odd cycle in $G$.
    \item $e$ is connected to an odd cycle in $G$ by a trivial edge.
    \item $e$ is an edge of an odd cycle in $G$.
\end{enumerate}
Furthermore, if (1) or (2) does not hold then $G$ contains a unique cycle.

\medskip

\noindent \textbf{Case 1.} $e$ is adjacent to an odd cycle in $G$. Let $C = 1,3,\ldots,2k$ be an odd cycle in $G$. We may assume that $\b_1 = a(\e_1 + \e_2)$, $\b_2 = \e_1 + e_3,\b_3 = \e_3 + \e_4, \ldots,\b_{2k} = \e_1 + \e_{2k}$ are the exponents corresponding to the edge $e$ and the odd cycle $C$. If $\beta_1 = 0$ then $f \in \overline{I(G\backslash e)^t} = I(G\backslash e)^t \subseteq I(G,\w)^t$. Hence, we may assume that $\beta_1 > 0$. Let $\beta_{2j+1} = \min\{ \beta_{2l+1} \mid 1 < l \le k-1\}.$ There are two subcases.

Case 1.a. $\beta_1 \le \beta_{2j+1}$. Let 

$$\gamma_{i}  = \begin{cases} \beta_i + \beta_{1} & \text{ if } i \text{ is even},\\
\beta_i - \beta_{1} & \text{ if } i \text{ is odd}.\end{cases}$$
Then we have $\sum_{i=1}^{2k} \gamma_i = \sum_{i=1}^{2k} \beta_i$ and 
$$\sum_{i=1}^s \gamma_i \b_i = \sum_{i=2}^s\beta_i \b_i + \beta_1 2 \e_1 \le \b.$$
Hence, we may reduce to the case $\beta_1 = 0$ and the conclusion follows.

Case 1.b. $\beta_1 > \beta_{2j+1}$. Let
$$\gamma_{i}  = \begin{cases} \beta_i + \beta_{2j+1} & \text{ if } i \text{ is even},\\
\beta_i - \beta_{2j+1} & \text{ if } i \text{ is odd}.\end{cases}$$
Then we have $\sum_{i=1}^s \gamma_i = \sum_{i=1}^s \beta_i$ and 
$$\sum_{i=1}^s \gamma_i \b_i = \sum_{i=2}^s\beta_i \b_i + (\beta_1 - \beta_{2j+1}) \b_1 +  2 \e_1 \beta_{2j+1} \le \b.$$
Since $\gamma_{2j+1} = 0$ and $G\backslash e_{2j+1}$ is compact, by induction on the number of edges and Lemma \ref{lem_criterion_integral_k}, we deduce that 
$$f \in \overline{I(G\backslash e_{2j+1},\w)^t}= I(G\backslash e_{2j+1},\w)^t \subseteq I(G,\w)^t.$$

\medskip

\noindent \textbf{Case 2.} $e = \{1,2\}$ is connected to an odd cycle in $G$ by a trivial edge. Since $(G,\w)$ has no $F_4$-minor, we deduce that the vertices of the connecting edge are the stems of $G$. In other words, we may assume that $\{1,3\} \in E(G)$ and $3$ is a stem of a bouquet of odd cycles in $G$. In particular, we have $C = 3,\ldots,2k+1$ is an odd cycle of $G$. We assume that $\b_1,\ldots,\b_{2k+1}$ are the exponents corresponding to the edges $\{1,2\},\{1,3\}$ and edges of $C$. Let $\beta_{2j+1} = \min\{ \beta_{2l+1} \mid 1 < l \le k-1\}.$ There are two subcases.

Case 2.a. $\beta_1 \le 2\beta_{2j+1}/a$. Let $\gamma=a \beta_1/2$, $\gamma_1=0$ and

$$\gamma_{i}  = \begin{cases} \beta_i + \gamma & \text{ if } i >2\text{ is even},\\
\beta_i - \gamma& \text{ if } i \text{ is odd}\\
\beta_i + 2 \gamma &\text{ if } i = 2.\end{cases}$$
Then we have $\sum_{i=1}^{2k+1} \gamma_i = \sum_{i=1}^{2k+1} \beta_i + \gamma -\beta_1$. Hence,
$$\sum_{i=1}^s \gamma_i \ge \sum_{i=1}^s \beta_i \text{ and } \sum_{i=1}^s \gamma_i \b_i = \sum_{i=1}^s\beta_i \b_i -a \beta_1  \e_2\le \b.$$
Since $\gamma_1 = 0$, by Lemma \ref{lem_criterion_integral_k}, we deduce that $f \in \overline{ I(G\setminus e)^t} = I(G\backslash e)^t \subseteq I(G,\w)^t$.

Case 2.b. $\beta_1 >2 \beta_{2j+1}/a$. Let
$$\gamma_{i}  = \begin{cases} \beta_i + \beta_{2j+1} & \text{ if } i >2\text{ is even},\\
\beta_i - \beta_{2j+1} & \text{ if } i >1\text{ is odd}\\
\beta_i +2\beta_{2j+1} & \text{ if } i =2\\
\beta_1-2\beta_{2j+1}/a & \text{ if } i =1.\end{cases}$$
Then we have 
\begin{align*}
    \sum_{i=1}^s \gamma_i & = \sum_{i=1}^s \beta_i + \beta_{2j+1} \left ( 1-\frac{2}{a} \right ) \ge \sum_{i=1}^s \beta_i \text{ and }\\
    \sum_{i=1}^s \gamma_i \b_i &= \sum_{i=1}^s\beta_i \b_i  -  2 \beta_{2j+1}\e_2 \le \b.
\end{align*}  
Since $\gamma_{2j+1} = 0$ and $G\backslash e_{2j+1}$ is compact, by induction on the number of edges and Lemma \ref{lem_criterion_integral_k}, we deduce that 
$$f \in \overline{I(G\backslash e_{2j+1},\w)^t}= I(G\backslash e_{2j+1},\w)^t \subseteq I(G,\w)^t.$$

\medskip

\noindent \textbf{Case 3.} $e$ is an edge of an odd cycle and $G$ is the odd cycle itself. This case follows from \cite[Theorem 4.9]{DZCL}. We give an alternative proof using our approach. We denote by $\b_1,\ldots,\b_{2k+1}$ be the exponents corresponding to minimal generators of $I(G,\w)$ corresponding to the edges of $C$ by clockwise order, where $\b_1 = a(\e_1 + \e_2)$ corresponds to $e$. In this case, we have $s = 2k+1$. If $\beta_j = 0 $ for some $j$ then we are done by induction on the number of edges. Hence, we may assume that $\beta_j > 0 $ for all $j$. Fix a pair of exponents $\b_{2j},\b_{2j+1}.$ Let $\gamma$ be any positive numbers less than or equal to $\beta_i$ for all $i\ge 2j$. Then we let $\gamma_1 = \beta_1 - \gamma/a$ and for all $i \ge 2j+1$, we set
$$\gamma_{i}  = \begin{cases} \beta_i + \gamma & \text{ if } i \text{ is odd},\\
\beta_i - \gamma & \text{ if } i \text{ is even}.\end{cases}$$
Then we have 
\begin{align*}
\sum_{i=1}^{2k+1} \gamma_i & = \sum_{i=1}^{2k+1} \beta_i + \gamma \left ( 1 - \frac{1}{a} \right ) \ge \sum_{i=1}^{2k+1} \beta_i \text{ and } \\
\sum_{i=1}^{2k+1} \gamma_i \b_i &= \sum_{i=1}^{2k+1}\beta_i \b_i   -  \gamma\e_2 + \gamma \e_{2j+1}.    
\end{align*}
Note that $b_{2j+1} = \beta_{2j} + \beta_{2+1}$. Hence, if we choose $\gamma \le \lceil  \beta_{2j} + \beta_{2j+1} \rceil - (\beta_{2j} + \beta_{2j+1})$ we deduce that we may replace $\beta_i$ by $\gamma_i$ and still get $\a \ge \sum_{i=1}^{2k+1} \gamma_i \beta_i$. Now, we start clockwise and apply these changes to get either $\gamma_{2j} + \gamma_{2j+1} = \lceil  \beta_{2j} + \beta_{2j+1} \rceil$ for all $j = 1, \ldots, k$ or $\gamma_i = 0$ for some $i$. If $\gamma_i=0$ for some $i$ then we are done by induction on the number of edges. If $\gamma_{2j}+\gamma_{2j+1} \ge 2$ then one of them is at least $1$ and we are done by induction on $t$. Hence, we assume that $\gamma_{2j}+\gamma_{2j+1}=1$ for all $j=1, \ldots,k$. Since $\sum_{j=1}^{2k+1} \gamma_j\ge t$, we deduce that $\gamma_1 \ge t-k.$ By induction on $t$ we may then reduce to the case that $\gamma_1=0 $. That completes the proof of Case 3 and the lemma.
\end{proof}

We are now ready for the proof of the main theorem.

\begin{proof}[Proof of Theorem \ref{thm_normality}] We assume that $\w$ is nontrivial. By Lemmas \ref{lem_F_1}, \ref{lem_F_2}, \ref{lem_F_3}, \ref{lem_F_4}, and \ref{lem_F_5}, we may assume that $(G,\w)$ has no $F_i$-minor for all $i = 1,\ldots,5$. We prove by induction on $|V(G)|$ and on $|E(G)|$ that $I = I(G,\w)$ is normal. For ease of reading, we divide the proof into several steps.

\medskip

\noindent \textbf{Step 1.} Reduction to the case where $G$ is connected. Assume that $G = G_1 \cup G_2$ is the disjoint union of two induced subgraphs $G_1$ and $G_2$. Since $G$ has no $F_3$-minor, we deduce that one of the weight functions $\w_1$ and $\w_2$ induced on the edges of $G_1$ and $G_2$ are trivial. We may assume that $\w_1$ is nontrivial and $\w_2$ is trivial. Since $(G,\w)$ has no $F_4$-minor, we deduce that $G_2$ is bipartite. By Lemma \ref{lem_drop_bipartite_component} and induction on $|V(G)|$, we may assume that $G$ is connected.

\medskip

\noindent \textbf{Step 2.} Reduction to the case where $(G,\w)$ has only one nontrivial edge. Let $f = x^\a$ be a minimal generator of $\overline{I^t}$. By Lemma \ref{lem_criterion_integral_k}, there exist nonnegative real numbers $\beta_1,\ldots,\beta_s$ and exponents $\b_1,\ldots,\b_s$ corresponding to minimal generators of $I$ such that $\sum_{j=1}^s \beta_j = t$ and $\a \ge \b = \sum_{j=1}^s \beta_j \b_j$. We may assume that $\b_1,\b_2$ are exponents in $\NP(I)$ corresponding to edges $e_1,e_2$ of $(G,\w)$ with nontrivial weights $\w(e_1) = a$ and $\w(e_2) = b$. There are two cases.  

\smallskip

Case 1. $e_1$ and $e_2$ have no common vertices. We may assume that $e_1 = \{1,2\}$ and $e_2 = \{3,4\}$. Since $(G,\w)$ has no $F_3$-minor, there is a trivial edge in $(G,\w)$ connecting $e_1$ and $e_2$. We may assume that $\{2,3\} \in E(G)$ is a trivial edge in $(G,\w)$ and $\b_3$ is the corresponding exponent in $\NP(I)$. Note that, $\b_1= a(\e_1+\e_2), \b_2=b(\e_3+\e_4), \b_3=\e_2+\e_3$. We may further assume that $\beta_1 a \le \beta_2 b$. Let $\gamma_1 = 0$, $\gamma_2 = (\beta_2 b - \beta_1a)/b$, $\gamma_3 = \beta_3 + \beta_1 a$. Then
\begin{align*}
    \c &= \gamma_1 \b_1 + \gamma_2 \b_2 + \gamma_3 \b_3 = \beta_2 \b_2- \beta_1a(\e_3+\e_4) +(\beta_3 + \beta_1a)(\e_2+\e_3)\\
    &=\beta_2\b_2+ \beta_3\b_3-\beta_1a\e_4 + \beta_1 a \e_2 \le \beta_1 \b_1 + \beta_2 \b_2 + \beta_3 \b_3.
\end{align*}
Furthermore, 
$$\gamma_1 + \gamma_2 + \gamma_3 = \beta_2 + \beta_3 + \beta_1(a - a/b) \ge \beta_1 + \beta_2 + \beta_3.$$
Since $\gamma_1 = 0$, by Lemma \ref{lem_criterion_integral_k}, we deduce that $x^\a \in \overline{I(G',\w')^t}$, where $G' = G \backslash e_1$ and $\w'$ is the induced weight function on the edges of $G'$. By Lemma \ref{lem_drop_non_trivial_edge}, we have that $(G',\w')$ has no $F_i$-minors for all $i=1,\ldots,5$. By induction on the number of edges, we deduce that $x^\a \in I(G',\w')^t \subseteq I^t$.

\smallskip

Case 2. $e_1$ and $e_2$ have a common vertex. We may assume that $e_1 = \{1,2\}$ and $e_2 = \{1,3\}$. Since $(G,\w)$ has no $F_1$-minor, we deduce that $e_3 = \{2,3\}$ is an edge of $G$ with $\w(e_3) = 1$. Note that, $\b_1= a(\e_1+\e_2), \b_2=b(\e_1+\e_3), \b_3=\e_2+\e_3$. As in case 1, we may further assume that $\beta_1 a \le \beta_2 b$. Let $\gamma_1 = 0$, $\gamma_2 = (\beta_2 b - \beta_1a)/b$, $\gamma_3 = \beta_3 + \beta_1 a$. Then
\begin{align*}
    \c &= \gamma_1 \b_1 + \gamma_2 \b_2 + \gamma_3 \b_3 = \beta_2 \b_2- \beta_1a(\e_1+\e_3) +(\beta_3 + \beta_1a)(\e_2+\e_3)\\
    &=\beta_2\b_2+ \beta_3\b_3-\beta_1a\e_1 + \beta_1 a \e_2 \le \beta_1 \b_1 + \beta_2 \b_2 + \beta_3 \b_3.
\end{align*} 
We may now proceed as in Case 1 to obtain the reduction.

\medskip

\noindent \textbf{Step 3.} Conclusion step. By Step 1, Step 2, and Lemmas \ref{lem_drop_non_trivial_edge}, \ref{lem_drop_trivial_even_cycle}, \ref{lem_drop_non_trivial_even_cycle}, we may assume that $G$ is compact with only one nontrivial edge. The conclusion follows from Lemma \ref{lem_normal_compact}.
\end{proof}

\subsection*{Acknowledgments} Guangjun Zhu is supported by the Natural Science Foundation of Jiangsu Province (No. BK20221353).


\begin{thebibliography}{2}



\bibitem[ANKRQ]{ANKRQ}
I. Al-Ayyoub, M. Nasernejad, K. Khashyarmanesh, L. G. Roberts, V. C. Qui\~nonez,
{\em Results on the normality of square-free monomial ideals and cover ideals under some graph operations}, Math. Scand. {\bf 127} (2021), https://doi.org/10.7146/math.scand.a-128963.



\bibitem[BH] {BH} 
A. Barnejee and T. H. Ha,  
{\it Integral closures of powers of sums of ideals}, J. Algebr. Comb. {\bf 58} (2023), 307–323. 



\bibitem[DZCL] {DZCL} 
S. Duan, G. Zhu, Y. Cui, and J. Li, 
{\it Integral closure and normality of edge ideals of some edge-weighted graphs}, arXiv:2308.06016.
 

\bibitem[DV] {DV} 
L. A. Dupont and R. H. Villarreal,
{\em Edge ideals of clique clutters of comparability graphs and the normality of monomial ideals}, Math. Scand. {\bf 106} (2010), 88--98.

\bibitem[Ei] {Ei} 
D. Eisenbud, 
{\it Commutative Algebra. With a View Toward Algebraic Geometry}, Graduate Texts in Mathematics {\bf 150} (1995), Springer-Verlag, New York.





\bibitem[HL]{HL}
H. T. Ha and K. N. Lin,
{\em Normal 0-1 polytopes}, SIAM J. Disc. Math. {\bf 29} (2015), 210--223.

\bibitem[HT] {HT} 
T. H. Ha and N. V. Trung, 
{\it Membership criteria and containments of powers of monomial ideals}, Acta Mathematica Vietnamica {\bf 44} (2019), 117--139.




\bibitem[MT]{MT}
D. H. Mau and T. N. Trung,
{\em Stability of associated primes and depth of integral closures of
powers of edge ideals,} arXiv:2108.01830.

 \bibitem[MV]{MV}
 N. C. Minh and T. Vu,
 {\em Integral closure of powers of edge ideals and their regularity}, J. Algebra {\bf 609} (2022), 120--144.


\bibitem[MVZ] {MVZ} 
N. C. Minh, T. Vu, and G. Zhu, 
{\it Integrally closed and normal edge ideals of weighted oriented graphs}, preprint.


\bibitem[OH]{OH}
H. Ohsugi, T. Hibi,
{\em Normal polytopes arising from finite graphs},
 J. Algebra {\bf 207} (1998), 409--426.


\bibitem[PS] {PS} 
C. Paulsen and S. Sather-Wagstaff,
{\it Edge ideals of weighted graphs}, 
J. Algebra Appl. {\bf 12} (2013), no. 5, 1250223.



\bibitem[Sch1] {Sch1} 
A. Schrijver, 
{\it Theory of linear and integer programming}, John Wiley \& Sons, Ltd., Chichester, 1986.

 
\bibitem[Sch2]{Sch2}
A. Schrijver,
{\em Combinatorial optimization. Algorithms and Combinatorics}. Springer-Verlag, Berlin, 2003.


\bibitem[SH]{SH} 
I. Swanson and C. Huneke, 
{\it Integral closure of ideals, rings, and modules}, London Mathematical Society Lecture Note Series {\bf 336} (2006), Cambridge, UK.


\bibitem[SVV] {SVV} 
A. Simis, W. Vasconcelos, and R.H. Villarreal,
{\it The integral closure of subrings associated to
graphs}, J. Algebra {\bf 199} (1998), 281--289.


\bibitem[WL]{WL} 
Z. Wang and D. Lu, 
{\it The edge rings of compact graphs}, arXiv:2309.07587v2.


\end{thebibliography}
\end{document}